\documentclass{amsart}

\usepackage{tikz}
\usetikzlibrary{calc,math}
\usetikzlibrary{positioning,arrows,calc}

\usepackage{amsmath,amssymb,amsthm}

\usepackage{enumitem}

\usepackage{xcolor}

\usepackage[all]{xy}

\newtheorem{theorem}{Theorem}[section]

\newtheorem{question}[theorem]{Question}

\newtheorem{thmx}{Theorem}

\theoremstyle{definition}
\newtheorem{definition}[theorem]{Definition}

\theoremstyle{remark}
\newtheorem{example}[theorem]{Example}

\newcommand{\Aut}{\operatorname{Aut}}
\newcommand{\AutI}{\Aut_{\text{I}}}
\newcommand{\Autcb}{\Aut_{\text{cb}}}
\newcommand{\Autc}{\Aut_{\text{c}}}

\newcommand{\Ker}{\operatorname{Ker}}

\newcommand{\G}{\mathcal{G}}

\usepackage{aliascnt}
\usepackage[pdfpagelabels]{hyperref}
\makeatletter
\@namedef{subjclassname@2020}{%
  \textup{2020} Mathematics Subject Classification}
\makeatother

\AtBeginDocument{%
   \def\MR#1{}
}

\begin{document}

\title[Finite groups represented by finite incidence geometries]{Every finite group is represented by a finite incidence geometry}
\author{Antonio D{\'\i}az Ramos}
\address{Departamento de \'Algebra, Geometr{\'\i}a y Topolog{\'\i}a, Universidad de M{\'a}laga, 29071-M{\'a}laga, Spain}
\email{adiazramos@uma.es\\
ORCID: \href{https://orcid.org/0000-0002-1669-1374}{0000-0002-1669-1374}}
\author{R\'emi Molinier}
\address{Institut Fourier, UMR 5582, Laboratoire de Mathématiques, Université Grenoble Alpes, CS 40700, 38058 Grenoble cedex 9, France}
\email{remi.molinier@univ-grenoble-alpes.fr\\
ORCID: \href{https://orcid.org/0000-0002-3742-5307}{0000-0002-3742-5307}}
%\thanks{The second author was partially supported by ???}
\author{Antonio Viruel}
\address{Departamento de \'Algebra, Geometr{\'\i}a y Topolog{\'\i}a, Universidad de M{\'a}laga, 29071-M{\'a}laga, Spain}
\email{viruel@uma.es\\
ORCID: \href{https://orcid.org/0000-0002-1605-5845}{0000-0002-1605-5845}}

%\subjclass[2020]{Primary: 20F29; Secondary:}

\begin{abstract}
We investigate the relationship between finite groups and incidence geometries through their automorphism structures. Building upon classical results on the realizability of groups as automorphism groups of graphs, we develop a general framework to represent pairs of finite groups $(G, H)$, where $H \trianglelefteq G$, as pairs of correlation--automorphism groups of suitable incidence geometries. Specifically, we prove that for every such pair $(G, H)$, there exists a finite incidence geometry $\Gamma$ satisfying that the pair  $(\operatorname{Aut}(\Gamma), \operatorname{Aut}_I(\Gamma))$  of correlation--automorphism groups of $\Gamma$ is isomorphic to $(G, H)$. Our construction proceeds in two main steps: first, we realize $(G, H)$ as the correlation and automorphism groups of an incidence system; then, we refine this system into a genuine incidence geometry preserving the same pair of automorphisms groups. We also provide explicit examples, including a family of geometries realizing $(S_n, A_n)$ for all $n \ge 2$.
\end{abstract}

\maketitle

\section{Introduction}

Finite groups are often difficult to understand, and a powerful tool for studying them is to examine their actions on various structures such as sets, vector spaces, and graphs. One may also wish to consider more geometric actions, such as those on Tits buildings \cite{Tits} or, more generally, on incidence geometries \cite{Bue1,Bue2}. 

An incident geometry is, in particular, an \emph{incidence system}, that is, a set of elements, each with a given type (interpreted as ``point", ``line", ``plane", etc.), together with an incidence relation between them, subject to the condition that two elements of the same type cannot be incident. An incidence system $\Gamma$ can be represented by a simple vertex-colored graph $\G$, called the \emph{incidence graph}, where the vertices correspond to the elements, the colors correspond to their types, and two vertices are adjacent if and only if the corresponding elements are incident. In fact, the incidence system is totally determined by its incidence graph and many of its properties can be deduced from those of the associated graph. In this paper, we will therefore work primarily with incidence graphs.

Then, an incident system $\Gamma$ is an \emph{incident geometry} if, in the setting of its incidence graph, every complete subgraph $K$ of the corresponding incidence graph $\G$ containing at most one element of each type is contained in a \emph{chamber}: a complete subgraph of $\G$ containing exactly one element of each type. More precise definitions are provided in \autoref{sec:def incidence} and in the textbook~\cite{Bue2}. 

On these geometrical structures, two different notions of automorphisms are usually considered: \emph{automorphisms} and \emph{correlations}. Both are permutations of the elements; however, while automorphisms must preserve both types and incidences, correlations preserve only incidences. These correspond, respectively, to color-preserving and non–color-preserving (colorblind) automorphisms of the associated incidence graph. In particular, every automorphism is a correlation, and the group $\AutI(\Gamma)$ of automorphisms of a incidence system $\Gamma$ defines a normal subgroup of the group $\Aut(\Gamma)$ of correlations of $\Gamma$. 

In a recent paper \cite{LST2025}, Leeman, Stokes, and Tranchida construed incident geometries for some finite groups. More precisely, they are interested in, for a given finite group $G$, finding an incident geometry $\Gamma$ such that the automorphism group of $G$ is isomorphic to the group of correlations of $\Gamma$ so that the inner automorphisms of $G$ correspond to the automorphisms of $\Gamma$. The authors hope that the combinatorial properties of such incident geometries are deeply connected to the algebraic structure of $G$ and its automorphisms. In \cite{LST2025}, they constructed several incident geometries for classical families of groups and pave the way for the search of incident geometries for any group. 

In this paper, we put this question in a more general framework and study the realizability of pair of groups as pairs correlations-automorphisms of incident geometries.

\begin{definition}\label{def:pair_of_groups}
A pair of groups is couple $(G,H)$ where $G$ is a group and $H$ is a normal subgroup of $G$. We say that two pairs $(G,H)$ and $(G',H')$ are \emph{isomorphic}, denoted $(G,H)\cong(G',H')$,  if there exists an isomorphism $\varphi\colon G\to G'$ such that $\varphi(H)=H'$.
\end{definition}

The main result of this paper is that every pair of finite groups $(G,H)$ is isomorphic to a pair $\big(\Aut(\Gamma),\AutI(\Gamma)\big)$ for some incident geometry $\Gamma$. 

\begin{thmx}\label{thmA}
Let $(G,H)$ be a pair of finite groups. There exists a finite incidence geometry $\Gamma$ such that $(\Aut(\Gamma),\AutI(\Gamma))\cong (G,H)$.
\end{thmx}

\begin{proof} Let $(G,H)$ be a pair of finite groups.
If $G$ is trivial, the incidence geometry with one single element will suffice. Assume now that $G$ is non trivial.
By \autoref{thm:realizing_inicedence_systems}, there exists a finite incidence system $\Gamma$ such that $(\Aut(\Gamma),\AutI(\Gamma))\cong (G,H)$ and such that all elements have degree at least 2, and $\Gamma$ has no flag of rank $3$ or more. Then by \autoref{thm:from incidence system to incidence geometry}, there exists an incidence geometry $\Gamma'$ such that $(\Aut(\Gamma'),\AutI(\Gamma'))\cong (\Aut(\Gamma),\AutI(\Gamma))\cong(G,H)$.
\end{proof}

The proof of \autoref{thmA} is inspired by the realizability of finite groups as groups of automorphisms of graphs~\cite{Fr39}. It is based on the construction given in \autoref{thm:realizing_inicedence_systems}, which provides an incidence system with the necessary properties to apply \autoref{thm:from incidence system to incidence geometry}. This latter theorem constructs, from an incident system satisfying some technical conditions on degrees and flags in the associated incident graph, an incident geometry with the same pair correlations-automorphisms. 

However, one may want to work with its own preferred incidence system which may not satisfy the hypotheses of \autoref{thm:from incidence system to incidence geometry}. One may then use \autoref{thm:improving incidence system} to refine the given incidence system so that it satisfies the desired hypotheses. As an example we construct, in Example \ref{ex:(Sn,An) incidence systems}, an incident geometry for the pair $(S_n,A_n)$ for every $n\geq 2$, where $S_n$ and $A_n$ are the symmetric group and alternating group on $n$ letters. This gives an example of an incident geometry for the automorphism group of the alternating group of $A_n$ as desired in \cite{LST2025}. 

The rest of the paper is divided into four sections. \autoref{sec:def incidence} gives the principal definition on incidence systems, incidence geometries and their automorphisms. In \autoref{sec:pair of incidence system} the first main step of the proof of~\autoref{thmA} is paved by giving, for every pair $(G,H)$, an incidence system $\Gamma$ with $(\Aut(\Gamma),\AutI(\Gamma))\cong (G,H)$. \autoref{sec:systems to geometries} give the remaining ingredients by explaining how to induce an incidence geometry from an incident system with the same pair correlations-automorphisms. Finally \autoref{sec:example} gives some examples to illustrate our construction and how one may compose with them in practice. 

\subsection*{Acknowledges}
The second author was partially supported by IRN MaDeF (CNRS). The third author is partially supported by grant PID2023-149804NB-I00 funded by MCIN/AEI/
10.13039/501100011033. All the authors gratefully acknowledge the financial support and the stimulating research environment provided by Institut Fourier at Grenoble, which made it possible to develop the results presented in this paper during their stay at the institution.

\section{Incidence graphs, incidence geometries and their automorphisms}\label{sec:def incidence}

Incidence systems are usually defined as follows.

\begin{definition}
An \emph{incidence system} is a tuple $\Gamma=(X,\ast,t,I)$ with $X$ a non-empty set of \emph{elements}, $I$ a finite set of \emph{types}, $t\colon X\to I$ a surjective map called \emph{type function} and $\ast$ a symmetric binary relation called \emph{incidence relation}, such that for every $x,y\in X$, if $x\neq y$ and $x\ast y$, then $t(x)\neq t(y)$. The number of types $|I|$ is then the \emph{rank} of the incidence system. Finally for every $x\in X$ the \emph{degree} of $x$ is $|\{y\in X\mid\ y\ast x\}|$. 
\end{definition}

However, one way to think of an incidence system $\Gamma=(X,\ast,t,I)$ is as a simple proper colored graph $\G=(V,E,I,t)$ with set of vertices $V=X$, palette of colors $I$ and such that, for all $x,y\in V$, $\{x,y\}\in E$ if and only if $x\ast y$. This graph is commonly called \emph{incidence graph} of the incidence system $(X,\ast,t,I)$. In this paper, we will mostly work from the graph point-of-view. 
We now define the notion of incidence geometry. Notice also that, since an element cannot be incident to itself, the degree of an element $x\in X$ corresponds to the degree of $x$ seen as a vertex in the incidence graph.

\begin{definition}
Let $(X,\ast,t,I)$ be an incidence system. 
A \emph{flag} of an incidence system is a subset $F\subset X$ of pairwise incident elements, i.e. such that for every $x,y\in F$, if $x\neq y$ then $x\ast y$. We define the \emph{rank} of $F$ as its size,  $|F|$, or equivalently, as the number of types occurring in the flag. From the incidence graph point-of-view, a flag corresponds to a clique, and we define the \emph{rank} of a clique as its size. We define a \emph{chamber} as a flag of rank $|I|$ or as a clique of rank $|I|$. An incidence system is then an \emph{incidence geometry} if every flag is contained in a chamber or, equivalently, if every clique is contained in a chamber.
\end{definition}

Let us look at possible automorphisms of these incidence geometries.

\begin{definition}
Let $\Gamma=(X,\ast,t,I)$ be an incidence system. A \emph{correlation} of $\Gamma$ is a permutation $f$ of $X$ such that for all $x,y\in X$, 
\[ 
x\ast y\Leftrightarrow f(x)\ast f(y)\qquad \text{and}\qquad t(x)=t(y)\Leftrightarrow t(f(x))=t(f(y)).
\]
We denote by $\sigma_f\in \Sigma_I$ the permutation of types induced by $f$, and define an \emph{automorphism} of $\Gamma$ as a correlation $f$ with $\sigma_f$ equal to the identity permutation, i.e., such that $t(f(x))=t(x)$ for all $x\in X$. We denote by $\Aut(\Gamma)$ the group of correlations of $\Gamma$ and by $\AutI(\Gamma)$ the subgroup of $\Aut(\Gamma)$ consisting of the automorphisms of $\Gamma$, and we say that $\Gamma$ is an $(\Aut(\Gamma),\AutI(\Gamma))$-incidence system.
\end{definition}

From the graph point-of-view, a correlation of a incidence system $\Gamma$ is just an automorphism of the corresponding incidence graph $\G$ which is ``colored-blind", i.e., such that if two vertices have the same color, then their images have the same color.  Hence, we will call their graph counterpart \emph{colorblind automorphisms} and we will write $\Aut_{\text{cb}}(\G)$ for the set of these automorphisms of the incident graphs. Also, an automorphism of $\Gamma$ corresponds to an automorphism of $\G$ which preserves the colors. We call these automorphisms \emph{color automorphisms} and we write $\Autc(\G)$ for the subgroup consisting of these color automorphisms. So, if $\Gamma$ is an incidence system with incidence graph $\G$, we have $\Aut(\Gamma)=\Autcb(\G)$ and $\AutI(\Gamma)=\Autc(\G)$.

Notice that $\AutI(\Gamma)$, or equivalently $\Autc(\G)$ for $\G$ the corresponding incidence graph, is actually a normal subgroup of $\Aut(\Gamma)$, respectively $\Aut_{\text{cb}}(\G)$), and the following question naturally arises.

\begin{question}
For every pair $(G,H)$ with $G$ a finite group and $H$ a normal subgroup of $G$, is there an incident geometry $\Gamma$  such that the pair $(\Aut(\Gamma),\AutI(\Gamma))$ is isomorphic to the pair $(G,H)$? 
\end{question}

 Given the pair $(G,H)$, this question is equivalent to the existence of a simple proper colored graph $\mathcal{G}$ such that $(\Autcb(\G),\Autc(\G))$ is isomorphic to the pair $(G,H)$.

\section{Finite groups and normal subgroups as automorphisms of incidence systems}\label{sec:pair of incidence system}

\begin{theorem}\label{thm:realizing_inicedence_systems}
    Let $G$ be a finite non trivial group and $H\leq G$ be a normal subgroup, so they define a pair $(G,H)$. Then there exists a finite incidence system $\Gamma$ such that
    \begin{enumerate}[label={\rm (\arabic{*})}]
        \item\label{thm:1_realizing_inicedence_systems} Every vertex in the incidence graph has degree at least $2$.
        \item\label{thm:2_realizing_inicedence_systems} All cliques in the incidence graph have rank at most $2$.
        \item\label{thm:3_realizing_inicedence_systems} $(\Aut(\Gamma),\AutI(\Gamma))$ is isomorphic to the pair $(G,H)$.
    \end{enumerate}
\end{theorem}
\begin{proof}
In this proof, we follow the ideas in \cite[Section 2]{Fr39}. We first introduce some notation:

Let $\mathbf{g}:=|G|$, and $\{g_1,\ldots,g_\mathbf{g}\}$ be the list of elements of $G$, assuming $g_1\in G$ is the identity element while $\{g_1,\ldots ,g_{[G:H]}\}$ are representatives of the left cosets $G/H$, so
    $$G=\operatornamewithlimits{\bigsqcup}_{j=1}^{[G:H]} g_jH.$$  
    
    Let   
    $\overrightarrow{\mathcal{G}}$ be the Cayley directed (and edge colored) graph as constructed in \cite[Chapter VIII, \S 5]{Konig} (see \cite[Schritt 1), pp.\ 242--243]{Fr39}), so
    \begin{align*}
        V(\overrightarrow{\mathcal{G}}) &=\{ P_i\,:\, i=1,\ldots, \mathbf{g}\}\\
        E(\overrightarrow{\mathcal{G}}) &=\{ (P_i,P_j)\,:\, i\ne j\}
    \end{align*}
    and the edge $(P_i,P_j)$ is labeled (colored) with the number $k$ for $g_k:=g_jg_i^{-1}$ (or equivalently, $g_j=g_kg_i$). So  $\overrightarrow{\mathcal{G}}$ is a symmetric complete directed graph on $\mathbf{g}$ vertices whose edges have been labeled with $\mathbf{g}$ colors.  The group of transformations of $\overrightarrow{\mathcal{G}}$ (preserving both direction and color of edges) is isomorphic to $G,$ where for any $g\in G$, $g$ gives rise to a colored automorphism of $\overrightarrow{\mathcal{G}}$ given by $g(P_i)=P_k$ where $g_k:=g_ig$. 

    We are now applying an arrow replacement to $\overrightarrow{\mathcal{G}}$ in order to obtain an undirected graph $\mathcal{G}$ while maintaining the group of automorphisms. So, if the edge $(P_i,P_j)$ in $\overrightarrow{\mathcal{G}}$ is labeled (colored) with the number $k\in\{2,3,…,\mathbf{g}\},$ we proceed as follows:
\begin{enumerate}
    \item We remove the edge $(P_i,P_j)$.
    \item We replace it with the following undirected graph
    
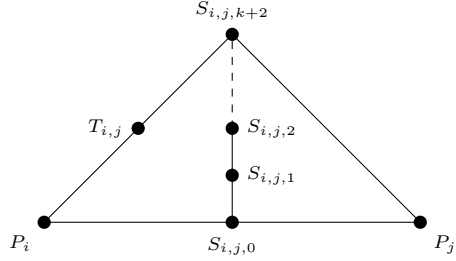
\begin{center}
\begin{tikzpicture}[scale=5, font=\scriptsize]

% Define the three main triangle vertices
\coordinate (Pi) at (0,0);       % P_i
\coordinate (Pj) at (1,0);       % P_j
\coordinate (Top) at (0.5,0.5);  % Right angle vertex (S_{i,j,k})

% Base midpoint (for S_{i,j,0})
\coordinate (BaseMid) at (0.5,0); 

% Draw triangle sides
\draw (Pi) -- (Pj) -- (Top) -- cycle;

% Draw nodes (small black circles)
\fill (Pi) circle (0.5pt);
\fill (Pj) circle (0.5pt);
\fill (Top) circle (0.5pt);
\fill (BaseMid) circle (0.5pt);

% Label P_i and P_j
\node[below left=2pt] at (Pi) {$P_i$};
\node[below right=2pt] at (Pj) {$P_j$};

%Label mid point
\node[below=2pt] at (BaseMid) {$S_{i,j,0}$};

%Label top
\node[above=2pt] at (Top) {$S_{i,j,k+2}$};

% Define and label the S_{i,j,n} vertices along the height 
\foreach \k in {1,2} {
  \pgfmathsetmacro{\y}{0.5*\k/4}
  \coordinate (S\k) at (0.5,\y);
  \fill (S\k) circle (0.5pt);
  \node[right=2pt] at (S\k) {$S_{i,j,\k}$};
}

% Draw height 
\draw (BaseMid) -- (S2);
\draw[dashed] (S2) -- (Top);

% Define and label T_{i,j} at midpoint of Pi and Top
\path let \p1 = (Pi), \p2 = (Top) in
  coordinate (Tij) at ({(\x1+\x2)/2}, {(\y1+\y2)/2});
\fill (Tij) circle (0.5pt);
\node[left=2pt] at (Tij) {$T_{i,j}$};

\end{tikzpicture}
    \end{center}
    
\end{enumerate}

Notice the arrow replacement we are applying here is different from the one described in \cite[Schritt 2), pp.\ 243--244]{Fr39} since we want to avoid leaves.    

Once all the oriented edges of $\overrightarrow{\mathcal{G}}$ has been replaced as above, we obtain a non-oriented simple graph $\mathcal{G}$ in which every vertex has degree at least $2$, and all cliques have rank at most $2$.

We now claim that $\Aut(\mathcal{G})\cong G.$ Indeed, notice that in $\mathcal{G}$, 
\begin{itemize}[label={--}]
    \item vertices $P_i$ have degree $4(\mathbf{g}-1)$,
    \item vertices $S_{i,j,0}$ and $S_{i,j,k+2}$, for $g_k:=g_jg_i^{-1}$, have degree $3,$ and
    \item all other vertices have degree $2$.
\end{itemize}
Therefore, given any $f\in\Aut(\mathcal{G})$, there must exists a permutation $\rho\in\Sigma_\mathbf{g}$ such that $f(P_i)=P_{\rho(i)}$ for every $i.$ 

Moreover, for any $j\ne i,$ $T_{i,j}$ is the only 
vertex of degree $2$ in the neighborhood of $P_i$ such that it is at distance $2$ of $P_j$, thus $f(T_{i,j})$ must be the only vertex of degree $2$ in the neighborhood of $P_{\rho(i)}$ such that it is at distance $2$ of $P_{\rho(j)},$ that is  $f(T_{i,j})=T_{\rho(i),\rho(j)}.$ 

Now, for any $j\ne i,$ the only vertex degree $3$ in the neighborhood of $T_{i,j}$ is $S_{i,j,k+2}$, where $g_k:=g_jg_i^{-1}$. Therefore $f(S_{i,j,k+2})=S_{\rho(i),\rho(j), k'+2},$ where $g_{k'}:=g_{\rho(j)}g_{\rho(i)}^{-1}.$ Moreover, $S_{i,j,0}$ is the only vertex degree $3$ which is connected to $S_{i,j,k+2}$ by a path of vertices of degree $2,$ and this path has length $k+2$. Hence $f(S_{i,j,0})=S_{\rho(i),\rho(j), 0},$ and $k'=k.$ In other words, the colors of the edges $(P_i,P_j)$ and $(P_{\rho(i)},P_{\rho(j)})$ in $\overrightarrow{\mathcal{G}}$ are equal. Therefore $\rho$ induces 
and automorphism of $\overrightarrow{\mathcal{G}}$ preserving both direction and color of edges, that is $\rho$ can be identified with an element of $G$ (indeed, $\rho\equiv (g_j\mapsto g_j\rho(1))$). This shows that $\Aut(\mathcal{G})\leq G.$  

Finally, every transformation of $\overrightarrow{\mathcal{G}}$ preserving both direction and color of edges gives rise to an automorphism of $\mathcal{G}$ since the arrow replacement described above codifies the direction and colors of edges. Indeed, given $g\in G$, then $g\colon \mathcal{G}\to \mathcal{G}$ is given by:
%\begin{equation}\label{eq:G_acts_on_Graph}
\begin{align}
g(P_i) &=P_k\text{ where $g_k:=g_ig,$}\nonumber\\ 
g(T_{i,j}) &=T_{k,l}\text{ where $g_k:=g_ig,$ and $g_l:=g_jg,$}\label{eq:G_acts_on_Graph}\\ 
g(S_{i,j,r}) &=S_{k,l,r}\text{ where $g_k:=g_ig,$ and $g_l:=g_jg.$} \nonumber  
\end{align}
%\end{equation}
Hence $\Aut(\mathcal{G})= G$.

We are now giving a coloring on the vertices of $\mathcal{G}$ to obtain the desired incidence system $\Gamma$ with associated incidence graph $\mathcal{G}$. So let $I:=\{ 0, 1,\dots, [G:H]+2\},$ and $t\colon V(\mathcal{G})\to I$ be given by
\begin{align*}
    t(P_i) &:= j,\text{ for $g_i\in g_jH$,}\\
    t(T_{i,j}) &:=0,\\
    t(S_{i,j,l}) &:= 
    \begin{cases}
        [G:H]+1,& \text{if $l$ is even,}\\
        [G:H]+2, & \text{if $l$ is odd.}
    \end{cases}
\end{align*}

The map $t$ gives rise to a proper coloring of $\mathcal{G}$ since every edge in this graph is of one of the following types:
\begin{itemize}[label={--}]
    \item Type $\{P_i, T_{a,b}\}$, where $a\ne b$ and $i\in\{a,b\}$. Then $t(T_{a,b})=0$ and $t(P_i)>0.$

    \item Type $\{P_i, S_{a,b,0}\}$, where $a\ne b$ and $i\in\{a,b\}$. Then $t(S_{a,b,0})=[G:H]+1$ and $t(P_i)\leq [G:H].$

    \item Type $\{T_{i,j}, S_{i,j,k+2}\}$ where $i\ne j$ and $g_k:=g_jg_i^{-1}$. Then $t(S_{i,j,k+2})>0$ and $t(T_{i,j})=0.$

    \item Type $\{S_{i,j,l}, S_{i,j,l+1}\}$ where $i\ne j$ and $l=0,\ldots,k+1$ for $g_k:=g_jg_i^{-1}$. Then $l$ and $l+1$ have different parity and therefore $t(S_{i,j,l})\ne t(S_{i,j,l+1}).$
\end{itemize} 

It remains to check that, if $\Gamma$ is the incidence system with associated incidence graph $\mathcal{G}$ and type function $t$, then  $(\Aut(\Gamma),\AutI(\Gamma))=(\Autcb(\G),\Autc(\G))$ is isomorphic to the pair $(G,H)$.

Observe that $\Autcb(\G)\leq \Aut(\G)=G$, so we just need to show that $G\leq \Autcb(\G).$ Now, given $g\in G,$ Equation \eqref{eq:G_acts_on_Graph} shows that $g$ maps $T$'s to $T$'s (both having color $0$) and $S_{i,j,r}$ to $S_{k,l,r}$ (both having the same color depending on the parity of $r$), so we just need to check how $g$ acts on the $P$'s and their colors: If $t(P_i)=t(P_k)=c$, then $g_i,g_k\in g_cH$, and since $H$ is normal, for $b\in\{1,\ldots,[G:H]\}$ defined by $g_cg\in g_bH$, we have that $g_j,g_l\in g_bH$ for $g_j:=g_ig$ and $g_l:=g_kg$. Therefore 
$$b=t(P_j)=t(P_l)=t(g(P_i))=t(g(P_j)).$$
Thus, $g$ is a colorblind automorphism of $\G$ and $\Autcb(\G)= \Aut(\G)=G$.

Finally, the discussion above shows that, in order to compute $\Autc(\G)$, we just need to determine the elements of $g$ that preserve the color given to vertices of type $P$. But $t(P_i)=t(g(P_i))=c,$ for every $i=1,\ldots,\mathbf{g}$ and $g_i\in g_cH$, holds if and only if $g_ig\in g_cH$, that is, if $g\in H.$ Thus $\Autc(\G)=H.$ 
\end{proof}

\section{From incidence systems to incidence geometries}\label{sec:systems to geometries}

Below, we construct, for a given incidence system realizing a pair of finite groups, an incidence geometry realizing the same pair. Then main underlying idea is to replace every edge by a chamber.

\begin{theorem}\label{thm:from incidence system to incidence geometry}
Let $\Gamma=(X,\ast,t,I)$ be a finite incidence system such that all elements have degree at least two and all flags have rank at most two, then there exists a finite incidence geometry $\Gamma'=(X',\ast',t', I)$ with $(\Aut(\Gamma'),\AutI(\Gamma'))\cong (\Aut(\Gamma),\AutI(\Gamma))$.
\end{theorem}

\begin{proof}
Denote by $\mathcal{G}=(V,E,I,t)$  the incidence graph corresponding to $\Gamma$. As commented before, the idea is to replace every edge by a chamber. More precisely, for every edge $\{v,w\}$ and each color $i\in I\setminus\{t(v),t(w)\}$, we add a vertex $u_{\{v,w\},i}$ and we set its color to $i$. In addition, we connect all these vertices among themselves as well as to $v$ and $w$, so that we get a chamber that contains $v$ and $w$. In this way, we construct an incidence graph $\mathcal{G'}=(V',E',I,t')$  with 
\begin{align*}
  V'&=V\sqcup\bigsqcup_{\{v,w\}\in E}\{u_{\{v,w\},i}\mid i\in I\setminus\{t(v),t(w)\}\}\text{, and}\\
  E'&=E\sqcup \bigsqcup_{e\in E}\{\{u_{e,i},u_{e,j}\}\mid i,j\in I, i\neq j\},
\end{align*}
where, in the equation above, we have identified, for every $e=\{v,w\}\in E$,
\begin{equation}\label{equ:nice_identification_of_vertices}
v=u_{\{v,w\},t(v)}\text{ and }w=u_{\{v,w\},t(w)}.
\end{equation} 
Finally, we define  $t'\colon V'\to I$ by $t'(v)=t(v)$ for all $v\in V$ and by $t'(u_{e,i})=i$ for all $e\in E$ and $i\in I$. The construction for $I=\{1,2,3,4,5\}$, $e=\{v,w\}\in E$, and $v$ and $w$ of colors $1$ and $2$ respectively, is depicted in the next figure,

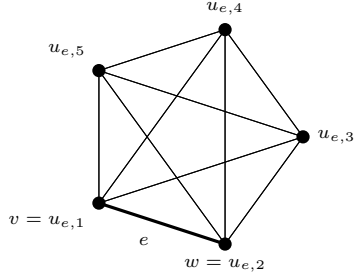
\begin{center}
\begin{tikzpicture}[scale=5, font=\scriptsize]

%define positions of five vertices on pentagon
\tikzmath{
\anglerot= 72;      %one fitht of a full loop
\radius=0.3;                 %radius to control size of diagram
coordinate \x;
\x1=({\radius*cos(1.0*\anglerot)},{\radius*sin(1.0*\anglerot)});
\x2=({\radius*cos(2.0*\anglerot)},{\radius*sin(2.0*\anglerot)});
\x3=({\radius*cos(3.0*\anglerot)},{\radius*sin(3.0*\anglerot)});
\x4=({\radius*cos(4.0*\anglerot)},{\radius*sin(4.0*\anglerot)});
\x5=({\radius*cos(5.0*\anglerot)},{\radius*sin(5.0*\anglerot)});
}

%five vertices on pentagon
\foreach \k in {1,2,3,4,5} {
    \fill (\x\k) circle (0.5pt);
}

%all segments for K5
\foreach \k in {1,2,3,4,5} {
    \foreach \l in {1,2,3,4,5} {
        \draw (\x\k) -- (\x\l);
    }
}

%a special segment
\draw[very thick] (\x3) -- (\x4);

%some labels
\node[above=2pt] at (\x1) {$u_{e,4}$};
\node[above left =2pt] at (\x2) {$u_{e,5}$};
\node[below left=2pt] at (\x3) {$v=u_{e,1}$};
\node[below=2pt] at (\x4) {$w=u_{e,2}$};
\node[right=2pt] at (\x5) {$u_{e,3}$};
\node[below left=2pt] at ($(\x3)!0.5!(\x4)$) {$e$}; %midpoint between x3 and x4

\end{tikzpicture}
\end{center}

In the remainder of the proof, for $e\in E$, we denote by $C(e)$ the subset $\{u_{e,i}\mid i\in I\}$. Since $|C(e)|=I$, this subset is a chamber. Since there was no flag of rank larger than two in $\Gamma$, the only cliques of rank larger than two in $\mathcal{G'}$ are contained in one $C(e)$ for some $e\in E$. In particular, the $C(e)$'s are the only chambers and every clique is contained in a chamber. Thus, the incident system $\Gamma'=(X',\ast',t',I)$ associated to $\mathcal{G'}$ is an incidence geometry. 

Consider now a colorblind automorphism $f$ of $\mathcal{G'}$. Thus, $f$ sends cliques to cliques and chambers to chambers.  Note also that, for all $v\in V$, since the degree of $v$ in $\mathcal{G}$ is at least two, the degree of $v$ in $\mathcal{G'}$ is at least $2|I|$. However, for $v\in V'\setminus V$, since $v$ is connected only to all vertices in a chamber, it has degree $|I|$. Therefore $f(V)=V$, $f$ induces a colorblind automorphism $f|_V\in \Autcb(\mathcal{G})$, and we have a group homomorphism 
\[\begin{array}{ccccc}
\Psi & : & \Autcb(\G') & \to & \Autcb(\G) \\
 & & f & \mapsto & f|_V. \\
\end{array}\]
Let $f\in \Ker(\Psi)$, so that $f|_V$ is the identity. In particular, $f$ fixes the colors in $\mathcal{G}$, which are the same as the colors in $\mathcal{G'}$. Also, since $f|_V$ is the identity, for every $e\in E$, $f(C(e))=C(e)$, as the $u_{e,i}$'s are the only vertices from $V'\setminus V$ that are connected to $v$ and $w$. As we have exactly one vertex of each color in a chamber and $f$ fixes the colors, $f$ is also the identity on $C(e)$ for every $e\in E$. Hence, $f$ is the identity and $\Psi$ is injective. 

Now, for $f\in \Autcb(\mathcal{G})$, we define $\widetilde{f}\colon V'\to V'$ by
\[
\widetilde{f}(u_{e,i})=u_{f(e),\sigma_f(i)},
\]
for $e\in E$ and $i\in I$, where $\sigma_f\in\Sigma_I$ is the permutation on $I$ induced by $f$, and where we have employed again the identification \eqref{equ:nice_identification_of_vertices}. Note that, by construction, for all $e\in E$, $\widetilde{f}(C(e))=C(f(e))$ and, in particular, $\widetilde{f}$ is a colorblind automorphism of $\G'$. Finally, for all $v\in V$, since the degree of $v$ is at least 2, we have $v=u_{\{v,w\},t(v)}$ for some $w\in V$, and hence
\[
\widetilde{f}(v)=f(u_{\{v,w\},t(v)})=u_{\{f(v),f(w)\},\sigma_f(t(v))}=u_{\{f(v),f(w)\},t(f(v))}=f(v).
\]
Therefore $\Psi(\widetilde{f})=f$ and we conclude that $\Psi$ is an isomorphism. Moreover, it is clear by construction that, for $f\in\Autcb(\G')$, $\sigma_f=\sigma_{\Psi(f)}$. It follows then that $\Psi(\Autc(\G'))=\Autc(\G)$ and we are done.
\end{proof}

\autoref{thm:from incidence system to incidence geometry}constructs a finite incidence geometry from a finite incidence system that shares the same pair correlations-automorphisms, provided the original geometry contains only flags of rank at most two and elements of degree at least two. The next result provides a method to refine any incidence system so that \autoref{thm:from incidence system to incidence geometry} can be applied.  

\begin{theorem}\label{thm:improving incidence system}
For any finite incidence system $\Gamma$, there exists a finite incidence system $\Gamma'$ with $(\Aut(\Gamma'),\AutI(\Gamma'))\cong (\Aut(\Gamma),\AutI(\Gamma))$ and such that all flags of $\Gamma'$ have rank at most two and all elements of $\Gamma'$ have degree at least two.
\end{theorem}
\begin{proof}
Let $\mathcal{G}=(V,E,I,t)$  be the incidence graph corresponding to $\Gamma$. Starting from $\mathcal{G}$, we will construct an incidence graph $\mathcal{G}'$ such that, if $\Gamma'$ is the incidence system associated to $\mathcal{G}'$, then $\Gamma'$ satisfies the conclusions of the statement.

Define $V_i$ as the subset of vertices of $V$ of degree $i$, set $V_{0,1}=V_0\cup V_1$, and consider the map $t\colon V\to I$ and the following natural number,
\[
M=\max_{i\in I}{|t^{-1}(\{i\})|},
\]
i.e., for every color $i\in I$, there is at most $M$ vertices of that same color. 

Next, we define a colored graph $\mathcal{G}'=(V',E',I',t')$ as follows,
\begin{align*}
V'&=V\sqcup \bigsqcup_{e\in E} \{u_{e,0},\ldots,u_{e,{2M+4}}\}\sqcup \bigsqcup_{w\in V_{0,1}} \{u_{w,0},\ldots,u_{w,{2M+4}}\},\\
E'&=\bigsqcup_{e=\{v,w\}\in E}\big\{\{v,u_{e,0}\},\{u_{e,0},w\},\{u_{e,0},u_{e,1}\},\{u_{e,1},u_{e,2}\},\ldots\\
&\ldots,\{u_{e,2M+3},u_{e,2M+4}\},\{u_{e,2M+4},u_{e,2M-1}\},\{u_{e,2M+3},u_{e,2M}\}\big\}\\
&\sqcup \bigsqcup_{w\in V_{0,1}} \big\{\{w,u_{w,0}\},\{u_{w,0},u_{w,1}\},\{u_{w,1},u_{w,2}\},\ldots\\
&\ldots,\{u_{w,2M+3},u_{w,2M+4}\},\{u_{w,2M+4},u_{w,2M-1}\},\{u_{w,2M+3},u_{w,2M}\}\big\},\\
&\sqcup \bigsqcup_{w\in V_0} \big\{\{w,u_{w,2}\}\big\},\\
I'&=I\sqcup \{i_0,i_1\},
\end{align*}
where the elements $u_{e,i}$ and $u_{v,i}$ are all different and do not belong to $V$, $i_0,i_1\notin I$, $t'(v)=t(v)$ for $v\in V$, and 
\[
t'(u_{e,j})=t'(u_{w,j})=\begin{cases} i_0&\text{if $j\equiv 0\pmod 2$},\\
i_1&\text{if $j\equiv 1\pmod 2$.}\\
\end{cases}
\]
Thus, we have replaced every edge $e=\{v,w\}$ of $E$ by its barycentric subdivision together with a ray of $2M$ vertices emanating from its midpoint, and with a figure eight involving $6$ vertices at the tail, and we have added to each vertex $w\in V_{0,1}$  a similar ray with and figure eight, plus an edge from $w$ to $u_{w,2}$ if $w\in V_0$. For instance, for $M=3$, we have, 
    
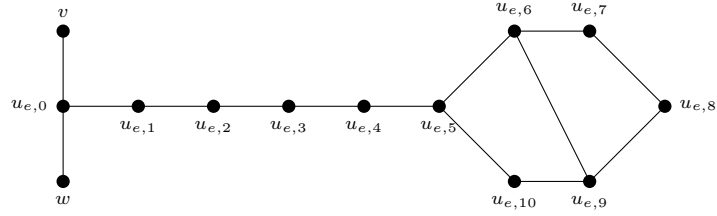
\begin{center}
\def\sl{0.2}%grid separation to control size of diagram
\begin{tikzpicture}[scale=5, font=\scriptsize]

%left-hand side vertical segment
\fill (0,\sl) circle (0.5pt);
\node[above=2pt] at (0,\sl) {$v$};
\draw (0,\sl)--(0,0)--(0,-\sl);
\fill (0,0) circle (0.5pt);
\node[left=2pt] at (0,0) {$u_{e,0}$};
\fill (0,-\sl) circle (0.5pt);
\node[below=2pt] at (0,-\sl) {$w$};

%central horizontal segment
\foreach \k in {1,2,3,4,5} {
    \fill (\k*\sl,0) circle (0.5pt);
    \node[below=2pt] at (\k*\sl,0) {$u_{e,\k}$};
    \draw (\k*\sl-\sl,0)--(\k*\sl,0);
}

%right-hand side figure eight
\fill (6*\sl,\sl) circle (0.5pt);
\node[above=2pt] at (6*\sl,\sl) {$u_{e,6}$};
\draw (5*\sl,0)--(6*\sl,\sl);
\fill (7*\sl,\sl) circle (0.5pt);
\node[above=2pt] at (7*\sl,\sl) {$u_{e,7}$};
\draw (6*\sl,\sl)--(7*\sl,\sl);
\fill (8*\sl,0) circle (0.5pt);
\node[right=2pt] at (8*\sl,0) {$u_{e,8}$};
\draw (7*\sl,\sl)--(8*\sl,0);
\fill (7*\sl,-\sl) circle (0.5pt);
\node[below=2pt] at (7*\sl,-\sl) {$u_{e,9}$};
\draw (8*\sl,0)--(7*\sl,-\sl);
\fill (6*\sl,-\sl) circle (0.5pt);
\node[below=2pt] at (6*\sl,-\sl) {$u_{e,10}$};
\draw (7*\sl,-\sl)--(6*\sl,-\sl)--(5*\sl,0);

\draw (7*\sl,-\sl)--(6*\sl,\sl);
  
\end{tikzpicture}
\end{center}
and
\begin{center}
\def\sl{0.2}%grid separation to control size of diagram
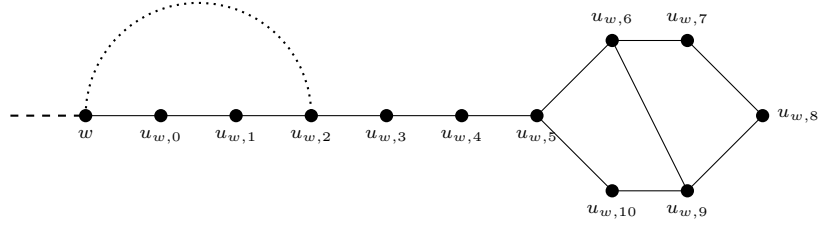
\begin{tikzpicture}[scale=5, font=\scriptsize]

%two left-most verteices
\draw[dashed,thick] (-2*\sl,0)--(-\sl,0);
\fill (-\sl,0) circle (0.5pt);
\node[below=2pt] at (-\sl,0) {$w$};
\draw (-\sl,0)--(0,0);
\draw[dotted,thick] (-\sl,0) arc[start angle=180, end angle=0, radius=1.5*\sl];
\fill (0,0) circle (0.5pt);
\node[below=2pt] at (0,0) {$u_{w,0}$};

%central horizontal segment
\foreach \k in {1,2,3,4,5} {
    \fill (\k*\sl,0) circle (0.5pt);
    \node[below=2pt] at (\k*\sl,0) {$u_{w,\k}$};
    \draw (\k*\sl-\sl,0)--(\k*\sl,0);
}

%right-hand side figure eight
\fill (6*\sl,\sl) circle (0.5pt);
\node[above=2pt] at (6*\sl,\sl) {$u_{w,6}$};
\draw (5*\sl,0)--(6*\sl,\sl);
\fill (7*\sl,\sl) circle (0.5pt);
\node[above=2pt] at (7*\sl,\sl) {$u_{w,7}$};
\draw (6*\sl,\sl)--(7*\sl,\sl);
\fill (8*\sl,0) circle (0.5pt);
\node[right=2pt] at (8*\sl,0) {$u_{w,8}$};
\draw (7*\sl,\sl)--(8*\sl,0);
\fill (7*\sl,-\sl) circle (0.5pt);
\node[below=2pt] at (7*\sl,-\sl) {$u_{w,9}$};
\draw (8*\sl,0)--(7*\sl,-\sl);
\fill (6*\sl,-\sl) circle (0.5pt);
\node[below=2pt] at (6*\sl,-\sl) {$u_{w,10}$};
\draw (7*\sl,-\sl)--(6*\sl,-\sl)--(5*\sl,0);

\draw (7*\sl,-\sl)--(6*\sl,\sl);
  
\end{tikzpicture}
\end{center}
where the dashed edge is present if and only if $w\in V_1$ and the dotted edge is present if and only if $w\in V_0$. Note also that
\begin{equation}\label{equ:large_t'_preimage}
 |t'^{-1}(i)|=\begin{cases} |t^{-1}(i)|\leq M &\text{if $i\neq i_0,i_1$},\\
(|E|+|V_{0,1}|)(M+3)&\text{if $i=i_0$},\\
(|E|+|V_{0,1}|)(M+2)&\text{if $i=i_1$}.
\end{cases}
\end{equation}
Moreover, as an incidence system has a non-empty set of elements by definition, we have $|E|+|V_{0,1}|>0$, $1\leq M<|t'^{-1}(i_0)|,|t'^{-1}(i_1)|$, and $|t'^{-1}(i_0)|\neq |t'^{-1}(i_1)|$. We also note that $\mathcal{G}'=(V',E',I',t')$ is a simple proper colored graph and that, by construction, $\mathcal{G}'$ has no clique of rank larger than two and has all elements of degree at least $2$.

Consider now a colorblind automorphism $f$ of $\mathcal{G}'$. Notice that, by \eqref{equ:large_t'_preimage}, the colors $i_0$ and $i_1$ are fixed. Thus, $f(V'\setminus V)=V'\setminus V$ and $f(V)=V$. Hence $f$ induces a colorblinded automorphism $f|_V$ of $\mathcal{G}$ and we have a group homomorphism
\[\begin{array}{ccccc}
\Phi & : & \Autcb(\G') & \to & \Autcb(\G) \\
 & & f & \mapsto & f|_V. \\
\end{array}\]

Let $f\in \Ker \Phi$. Since $f|_V$ is the identity and $f$ fixes the color $i_0$ then, for every $e\in E$, we must have $f(u_{e,0})=u_{e,0}$ and, for every $w\in V_{0,1}$, we must have $f(u_{w,0})=u_{w,0}$. As $f$ also fixes the color $i_1$ and the graph consisting of each ray and its corresponding figure eight has no non-trivial color automorphism, we also must have $f(u_{e,j})=u_{e,j}$ and $f(u_{w,j})=u_{w,j}$ for all $j\in\{0,1,\dots,2M+4\}$. Hence $f$ is the identity and we conclude that $\Phi$ is injective. 

Now let $f\in\Autcb(\mathcal{G})$ and consider $\widetilde{f}\colon V'\to V'$ defined, for all $v\in V$, by $\widetilde{f}(v)=f(v)$, and for all $e\in E$, $w\in V_{0,1}$, and $j\in\{0,1,\dots,2M+4\}$, by $\widetilde{f}(u_{e,j})=u_{f(e),j}$ and $\widetilde{f}(u_{w,j})=u_{f(w),j}$. Then, by construction, $\widetilde{f}$ is a graph automorphism. Moreover, if $\sigma_f\in \Sigma_I$ is the permutation of the colors $I$ induced by $f$, then $\widetilde{f}$ induces the permutation of colors $\sigma_{\widetilde{f}}\in \Sigma_{I'}$ that extends $\sigma_f$ by $i_j\mapsto i_j$ for $j=0,1$. Hence $\widetilde{f}\in\Autcb(\mathcal{G}')$ and, by construction, $\Phi(\widetilde{f})=f$. So $\Phi$ is an isomorphism. 

Finally, note that, for $f\in \Autcb(\G')$, we have that $\sigma_{\Phi(f)}$ is the restriction to $I$ of $\sigma_f$ and that $\sigma_f(i_j)=i_j$ for $j=0,1$. Hence, $\sigma_f$ is the identity permutation on $I'$ if and only if $\sigma_{\Phi(f)}$ is the identity permutation on $I$. Therefore $\Phi(\Autc(\Gamma'))=\Autc(\Gamma)$.
% _________________________________________

% Next, by \eqref{equ:large_t'_preimage} and the definition of $M$, every colorblind or color automorphism $f$ of $\Gamma'$ must send type $i_j$ to type $i_j$ for $j=0,1,2$. In addition, as the $u_{e,0}$'s with $e\in E$ are the only vertices of type $i_0$ with degree $3$, they must be permuted among themselves. We deduce that, in fact, we have 
% \begin{equation}\label{eqref:gamma_on_raytail}
% f(u_{e,j})=u_{f(e),j}
% \end{equation}
% for all $e\in E$ and all $0\leq j\leq 3M+5$, where, for $\{v,w\}\in E$, we define $f(e)=\{f(v),f(w)\}$. In the opposite direction, every colorblind or color automorphism of $\Gamma$ can be extended in a unique way to $\Gamma'$ using the formula \eqref{eqref:gamma_on_raytail}. Thus we have
% \[
% (\Autcb(\Gamma),\Autc(\Gamma))\cong (\Autcb(\Gamma'),\Autc(\Gamma'))
% \]
% and we are done.
\end{proof}

\section{Examples}\label{sec:example}

In this section we ilustrate some of the construction an arguments given in the previous section. 

The following example shows that in general, it is not possible to obtain an incidence geometry by just adding vertices to the existing flags (so they became part of a chamber) while keeping fixed the group of correlations. Therefore, hypotheses and arguments given in \autoref{thm:from incidence system to incidence geometry} are the best possible.

\begin{example}\label{ex:completing_just_flags_is_not_enough}
Let $\Gamma$ be the $(S_3,A_3)$-incidence system described by the vertices connected by solid lines in \autoref{fig:example_S3_A3}. If we add black vertices and edges (the dashed lines) to the rank $2$ maximal flags so these flags become part of a chamber, we obtain a $(D_{12},C_3)$-geometry. Notice that the new black vertices give rise to a new correlation given by a rotational symmetry of order $6$, so increasing the group of correlations.

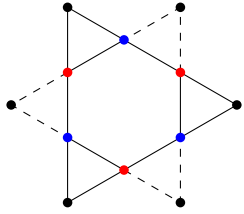
\begin{figure}[ht]
\centering

\begin{tikzpicture}[scale=3] % No global node style! Scale can be increased if needed

% Outer vertices (triangle tips) – black
\def\R{1/2} % Outer radius

\foreach \i in {0,1,2,3,4,5} {
  \pgfmathsetmacro{\angle}{60*\i}
  \coordinate (V\i) at ({\R*cos(\angle)}, {\R*sin(\angle)});
  \fill[black] (V\i) circle (0.6pt); % black outer nodes
}

% Define inner hexagon vertices (intersection points) 
\def\rInner{sqrt(3)/6} % Correct radius for inner hexagon

\foreach \i in {0, 1, 2, 3, 4, 5} {
  \pgfmathsetmacro{\angle}{60*\i + 30}
  \coordinate (H\i) at ({\rInner*cos(\angle)}, {\rInner*sin(\angle)});
}

% Draw the (S_3,A_3)-incidence system
\draw (V0) -- (V2) -- (V4) -- cycle;  %  triangle
\draw (H1) -- (H2);
\draw (H3) -- (H4);
\draw (H5) -- (H0); %hexagon

% Complete the flags to chambers to obtain a (D_12,A_3)-incidence system
\draw[dashed] (H0) -- (V1) -- (H1);
\draw[dashed] (H2) -- (V3) -- (H3);
\draw[dashed] (H4) -- (V5) -- (H5);

%Draw the vertices in the hexagon with the right color
\foreach \i/\color in {0/red, 1/blue, 2/red, 3/blue, 4/red, 5/blue} {
    \fill[fill=\color] (H\i) circle (0.6pt);
}
\end{tikzpicture}
\caption{$(S_3,A_3)$-incidence system and $(D_{12},C_3)$-geometry}
\label{fig:example_S3_A3}
\end{figure}

\end{example}

The following example demonstrates that, for certain pairs of finite groups $(G,H)$, ad-hoc constructions can yield incidence systems that are strictly smaller than those produced by \autoref{thm:realizing_inicedence_systems}, thereby offering a more efficient realization in specific cases.

\begin{example}\label{ex:(Sn,An) incidence systems}
   We use an inductive argument to construct an $(S_n,A_n)$-incidence system $\Gamma_n=(X_n,*,t_n,I_n)$ for $n\geq 2$ such that 
    \begin{enumerate}[label={\rm (\arabic{*})}]
        \item\label{item:Ex_SnAn_1} Every vertex in the incidence graph associated to $\Gamma_n$ has degree at least $2$.
        \item\label{item:Ex_SnAn_2} All cliques in the incidence graph associated to $\Gamma_n$ have rank at most $2$.
        \item\label{item:Ex_SnAn_3} $I_n=\{0,1,\ldots,n\}$ and 
        \[
        |t_n^{-1}(i)|=
        \begin{cases}
            n & \text{if $i=n,$}\\
            \frac{n!}{2\cdot (i-1)!} & \text{if $0\leq i<n$,}\\
        \end{cases}
        \]
        where $(-1)!=0!=1$. In particular, $|t_n^{-1}(n-1)|=\binom{n}{2}$.
        \item\label{item:Ex_SnAn_4} $\Aut(\Gamma_n)$ acts on $t_n^{-1}(n)$, it does so as the usual permutation action on a set of $n$ elements, and this action fully determines the action on $\Gamma_n$, i.e., the induced map $\Aut(\Gamma_n)\to \operatorname{Bijections}(t_n^{-1}(n))$ is injective.
    \end{enumerate}   

    For $n=2$, we define our pre-induction object $\Gamma_2$ as the $(S_2,A_2)$-incidence system associated with the colored graph in \autoref{fig:Gamma_2},  where the label inside the vertex indicates the color. Observe that there are just two vertices with color/type $2$.
    
\begin{figure}[ht]
    \centering
    \begin{tikzpicture}[scale=1.5, font=\scriptsize]

% Circle radius
\def\radius{4pt}

% Define positions of 4 vertices in a line
\foreach \i/\label in {0/2, 1/0, 2/1, 3/2} {
    \coordinate (V\i) at (\i, 0);
    \draw[thick] (V\i) circle (\radius); % draw large hollow circle
    \node at (V\i) {\label};             % place color label inside
}

% Draw edges between adjacent vertices
\foreach \i in {0,1,2} {
    \pgfmathtruncatemacro{\j}{\i+1}
    \draw[thick, shorten >=6pt, shorten <=6pt] (V\i) -- (V\j);
}
\end{tikzpicture}
\caption{The $(S_2,A_2)$-incidence system $\Gamma_2$}\label{fig:Gamma_2}
\end{figure}
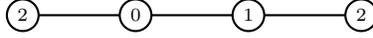

Assume that, for $n\geq 3$, the $(S_{n-1},A_{n-1})$-incidence system 
\[
\Gamma_{n-1}=(X_{n-1},*,t_{n-1},I_{n-1})
\]
satisfies the properties \ref{item:Ex_SnAn_1}, \ref{item:Ex_SnAn_2}, \ref{item:Ex_SnAn_3}, and \ref{item:Ex_SnAn_4} above. We define $\Gamma_n=(X_n,*,t_n,I_n)$ as the incidence system associated to the  colored graph constructed as follows: First, we consider the complete graph whose  vertices are the left cosets $A_n/A_{n-1}$. More precisely, we choose representatives $g_i\in A_n$ such that
\begin{equation}\label{equ:choice_of_representatives_An}
\text{$A_n=\sqcup_{i=1}^n g_iA_{n-1}$ and $g_n$ equals the identity element},
\end{equation}
and then define $v^n_i$ as the coset $g_iA_{n-1}$, so that $A_n/A_{n-1}=\{v^n_1,\ldots, v^n_n\}$. We note that this set of vertices is an $A_n$-(left) set and we assign color $n$ to all these vertices and define $I_n=I_{n-1}\sqcup\{n\}=\{0,1,\ldots,n\}$. 

%Note that this set is an $A_n$-set given by the left cosets $A_n/A_{n-1}$.  so that there exists $g_i\in %A_n,$ assuming $g_n$ is the identity element, such that $A_n=\sqcup_{i=1}^n g_iA_{n-1}$, and $v^n_i$ %represents the coset $g_iA_{n-1}$. 

Second, we add new vertices $v^n_{i,j}=v^n_{j,i}$ for $i\ne j,$, thus splitting every original edge $\{v^n_i,v^n_j\}$ into two new edges $\{v^n_i,v^n_{i,j}\}$ and $\{v^n_{i,j},v^n_j\}$. We assign color $(n-1)$ to all these vertices $v^n_{i,j}.$ 

Third, we add a copy $\Gamma^n_{n-1}$of $\Gamma_{n-1}$ by identifying the vertex $v^{n-1}_i\in X_{n-1}$ (which has assigned color $(n-1)$) with the vertex $v^n_{i,n}$.

Finally, we add $(n-1)$ more copies of $\Gamma_{n-1}$ by propagating this construction to the remaining vertices $v^n_j,$ $1\leq j<n$, via the action of $A_n$. More precisely, for such an index $j$, let $g\in A_n$ satisfy that $g(v^n_n)=v^n_j$. Then we add a copy $\Gamma^j_{n-1}$ of $\Gamma_{n-1}$ in the neighborhood of the vertex $v^n_j$ by identifying the vertex $v^{n-1}_i\in X_{n-1}$ with the vertex $v^n_{g(i),j}$ for $g(v^n_i)=v^n_{g(i)}$. Note that the names of the vertices of such a copy of $\Gamma_{n-1}$ may depend on the chosen elements $g$'s but, on the contrary, their types do not depend by Equation \eqref{equ:choice_of_representatives_An} and the induction hypothesis that the automorphism group of $\Gamma_{n-1}$ is $A_{n-1}$.

By construction, it is clear that $\Gamma_n$ satisfies points \ref{item:Ex_SnAn_1} and  \ref{item:Ex_SnAn_2}. Moreover, $\Gamma_n$ has $n$ vertices of type ``$n$", so that $|t_n^{-1}(n)| = n$, we have added $n$ copies of $\Gamma_{n-1}$ in which we identify pairs of vertices of type ``$n-1$", hence $|t_n^{-1}(n-1)| = \binom{n}{2}$, and we clearly have $|t_n^{-1}(i)| = n \cdot |t_{n-1}^{-1}(i)|$ for $0\leq i \leq  n - 2$. Thus, $\Gamma_n$ also satisfies point \ref{item:Ex_SnAn_3}. Finally, $\Aut(\Gamma_n)$ acts on $t_n^{-1}(n)$ by examination of Figure \ref{fig:Gamma_3} if $n=3$, or by comparing the sizes described in point (3) if $n\geq 4$. Moreover, this action clearly determines the action of $\Aut(\Gamma_n)$ on the $\binom{n}{2}$ vertices of type $n-1$.  Then, by the induction hypothesis that point \ref{item:Ex_SnAn_4} holds for $\Gamma_{n-1}$, the action of $\Aut(\Gamma_n)$ on $t_n^{-1}(n)$ also determines the action of $\Aut(\Gamma_n)$ on the $n$ copies of $\Gamma_{n-1}$ that build up $\Gamma_n$. Thus, point \ref{item:Ex_SnAn_4} also holds for $\Gamma_n$ and, in particular, 
\[
\AutI(\Gamma_n)\leq\Aut(\Gamma_n)\leq S_n.
\]

\begin{figure}[ht]
    \centering
\begin{tikzpicture}[scale=1.5, font=\scriptsize]
% Circle radius for vertex
\def\radius{4pt}

% Triangle tip vertices (equilateral triangle)
\coordinate (A) at (90:2);    % Top
\coordinate (B) at (210:2);   % Bottom left
\coordinate (C) at (330:2);   % Bottom right

% Outer vertices (color "3")
\foreach \point in {A,B,C} {
  \draw[thick] (\point) circle (\radius);
}
\node at (A) {3};
\node at (B) {3};
\node at (C) {3};

% Midpoints of edges (color "2")
\path (A) -- (B) coordinate[midway] (AB);
\path (B) -- (C) coordinate[midway] (BC);
\path (C) -- (A) coordinate[midway] (CA);

\foreach \point/\label in {AB/2, BC/2, CA/2} {
  \draw[thick] (\point) circle (\radius);
  \node at (\point) {\label};
}

% Vertices  of color "0" and "1")
\path (AB) -- (BC) coordinate[pos=1/3] (ABC1);
\path (AB) -- (BC) coordinate[pos=2/3] (ABC2);
\path (BC) -- (CA) coordinate[pos=1/3] (BCA1);
\path (BC) -- (CA) coordinate[pos=2/3] (BCA2);
\path (CA) -- (AB) coordinate[pos=1/3] (CAB1);
\path (CA) -- (AB) coordinate[pos=2/3] (CAB2);

\foreach \point/\label in {ABC1/0, ABC2/1, BCA1/0, BCA2/1,CAB1/0, CAB2/1} {
  \draw[thick] (\point) circle (\radius);
  \node at (\point) {\label};
}
 %Edges
 \foreach \v/\w in {A/AB, AB/B, B/BC, BC/C, C/CA, CA/A, AB/ABC1, ABC1/ABC2, ABC2/BC, BC/BCA1, BCA1/BCA2, BCA2/CA, CA/CAB1, CAB1/CAB2, CAB2/AB} {
 \draw[thick, shorten >=6pt, shorten <=6pt] (\v) -- (\w);
 }
\end{tikzpicture}
\caption{The $(S_3,A_3)$-incidence system $\Gamma_3$. Observe that the inner triangle contains $3$ copies of $\Gamma_2$.}\label{fig:Gamma_3}
\end{figure}
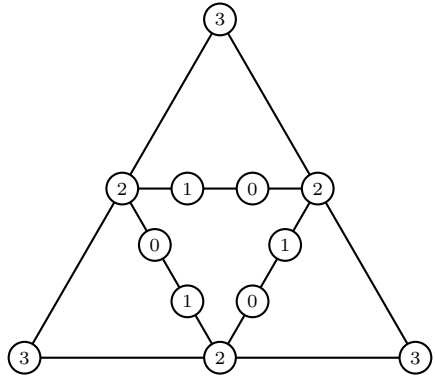

It is left to show that $\Gamma_n$ is indeed a $(S_n,A_n)$-incidence system and we do so next. The case $n=2$ is a straightforward calculation and, for $n>2$, we observe that:
\begin{enumerate}[label={\rm (\alph{*})}]

    \item By Equation \eqref{equ:choice_of_representatives_An} and because the automorphism group of $\Gamma_{n-1}$ is $A_{n-1}$, the group of automorphisms of this incidence system acts transitively on the set of vertices of type ``$n$" and contains the group $A_n$ acting in the standard way, thus the group of correlations does so. Therefore
    $$A_n\leq \AutI(\Gamma_n)\leq\Aut(\Gamma_n)\leq S_n.$$

    \item Given a transposition $\sigma=(i,j)\in S_n$, $\sigma$ induces a non-automorphic correlation in the copies $\Gamma^k_{n-1}$ with $k\neq i,j$ as, on those copies, $\sigma$ induces a transposition on the the vertices of type $n-1$ of $\Gamma_{n-1}^k$. (Note that, as $n\geq 3$, there is such a copy $\Gamma_{n-1}^k$.) In particular, $\sigma$ induces a non-automorphic correlation in $\Gamma_n$, i.e., 
    $$A_n\leq \AutI(\Gamma_n)\lneq\Aut(\Gamma_n)= S_n.$$
\end{enumerate} 
\end{example}

\bibliographystyle{abbrv}
\bibliography{biblio} %I have extended Remi's file

\end{document}